\documentclass[10pt,reqno]{amsart}
\usepackage[T2A]{fontenc}
\usepackage{amsmath}
\usepackage{amssymb}
\usepackage{amsfonts}
\usepackage{stmaryrd} 

\newtheorem{thm}{Theorem}
\newtheorem{prop}[thm]{Proposition}

\newtheorem{cor}[thm]{Corollary}

\begin{document}

\title[On the behavior of $\pi$-submaximal subgroups]{On the behavior of $\pi$-submaximal subgroups\\ under homomorphisms}
\author{Danila O. Revin}%
\address{Danila O. Revin
\newline\hphantom{iii} Sobolev Institute of Mathematics,
\newline\hphantom{iii} 4, Koptyug av.
\newline\hphantom{iii} 630090, Novosibirsk, Russia
} \email{revin@math.nsc.ru}

\author{Andrei V. Zavarnitsine}%
\address{Andrei V. Zavarnitsine
\newline\hphantom{iii} Sobolev Institute of Mathematics,
\newline\hphantom{iii} 4, Koptyug av.
\newline\hphantom{iii} 630090, Novosibirsk, Russia
} \email{zav@math.nsc.ru}

\maketitle {\small
\begin{quote}
\noindent{{\sc Abstract.} We construct examples that show the difference in behavior or $\pi$-maximal and $\pi$-submaximal subgroups under group homomorphisms.}


\end{quote}
}
\section{Introduction}

\thispagestyle{empty}

Let $\pi$ be a set of primes and let $G$ be a finite group. A subgroup $H\leqslant G$ is a {\em $\pi$-subgroup} if every prime divisor of $|H|$ is contained in~$\pi$. We
say that $H$ is {\em $\pi$-maximal}, if it is a $\pi$-subgroup and is not contained in a strictly larger $\pi$-subgroup. In the particular case $\pi=\{p\}$, the
$\pi$-maximal subgroups are precisely the $p$-Sylow subgroups.

The $\pi$-maximal subgroups have been extensively studied in the literature, see, for example, \cite{18GuoRev,18GuoRev1,18GuoRev2,94Wie511,03She,67Wie,80Wie,94Wie607}, \cite[Ch.~5, \S~3]{86Suz}. The most
well-studied are the $\pi$-maximal subgroups in solvable groups. According to P.\,Hall's classical theorem \cite{28Hall}, \cite[Ch.~4, Theorem~5.6]{86Suz}, every two
$\pi$-maximal subgroups of a solvable group $G$ are conjugate, or, equivalently, a complete analog of Sylow's theorem for $\pi$-subgroups holds. More precisely, a
solvable group $G$ has a $\pi$-subgroup whose index is not divisible by the primes in $\pi$ (such subgroups are called {\em $\pi$-Hall} and they are necessarily
$\pi$-maximal), every two $\pi$-Hall subgroups of $G$ are conjugate, and every $\pi$-subgroup is contained in a $\pi$-Hall subgroup. For nonsolvable group, the analog of
Hall's theorem does not hold. However, the study of $\pi$-Hall subgroups in nonsolvable groups is substantially facilitated by the fact that if $A$ is a normal subgroup
of $G$ then, for every $\pi$-Hall subgroup $H$ of $G$, the subgroups $H\cap A$ and $HA/A$ are $\pi$-Hall in $A$ and $G/A$, respectively \cite[Ch.~4, (5.11)]{86Suz}.

The principal difficulty in the study of $\pi$-maximal subgroups of nonsolvable groups is that the $\pi$-maximality agrees poorly with the normal structure of the group.
For example \cite[Example 2, p. 170]{86Suz}, a $2$-Sylow subgroup $H$ in $G=\operatorname{PGL}_2(7)$ is maximal, see \cite{85Atlas}, hence is $\{2,3\}$-maximal. However,
the normal subgroup $A=\operatorname{PSL}_2(7)\cong \operatorname{GL}_3(2)$ of order $2^3\cdot3\cdot7$ in $G$ is such that the intersection $H\cap A$ is a $2$-Sylow
subgroup of $A$ (and  $|H\cap A|=2^3$), but is not a $\{2,3\}$-maximal subgroup of $A$, because the stabilizers in $\operatorname{GL}_3(2)$ of a line and a plane of
the natural $3$-dimensional module both have order $2^3\cdot3$, i.\,e. are $\{2,3\}$-subgroups, and $H\cap A$ is conjugate to a $2$-Sylow subgroup of each of the
stabilizers by Sylow's theorem.

Homomorphic images of $\pi$-maximal subgroups exhibit an even more irregular behavior which we will talk about below. As far as intersections of $\pi$-maximal and normal
subgroups are concerned, not every $\pi$-subgroup of a normal subgroup $A$ can be such an intersection. An obstruction here is the Wielandt--Hartley theorem
\cite[Hauptsatz~13.2]{94Wie607}, \cite[Lemmas 2 and~3]{71Har}, \cite[Ch.~5, (3.20)]{86Suz} which states that if $H$ is a $\pi$-maximal subgroup of $G$ and
$A\trianglelefteqslant G$ then the order of $N_A(H\cap A)/(H\cap A)$ is not divisible by the primes in $\pi$; in particular, $H\cap A\ne 1$ if $|A|$ is divisible by at
least one prime in $\pi$. Later Wielandt announced\footnote{Unfortunately, a proof of this analog has not been published to this day.
In 2003, Shemetkov obtained the following weaker fact \cite[Theorem 7]{03She}: if $A\trianglelefteqslant\trianglelefteqslant G$
and $H$ is $\pi$-maximal in $G$ then $H\cap A\ne 1$ if $|A|$ is divisible by at
least one prime in $\pi$. Recently, Skresanov has found a proof of Wielandt's claim in its entire form \cite[5.4(a)]{80Wie} and
is currently preparing it for publication.}\label{skr}
an analog \cite[5.4(a)]{80Wie} of this theorem for the case where $A\trianglelefteqslant\trianglelefteqslant G$. It
allowed him to introduce the new concept of a $\pi$-submaximal subgroup.

According to Wielandt \cite[Definition on p.~170]{80Wie}, we say that $H$ is {\em $\pi$-submaximal} in $G$ if there exists a finite group $G^*$ such that
$G\trianglelefteqslant \trianglelefteqslant G^*$, i.\,e. $G$ can be embedded into $G^*$ as a subnormal subgroup, and $H=G\cap K$ for some $\pi$-maximal subgroup $K$ of
$G^*$.

Clearly, every $\pi$-maximal subgroup is $\pi$-submaximal, but the converse does not hold in general. For example, a $2$-Sylow subgroup
of $\operatorname{GL}_3(2)$ is $\{2,3\}$-submaximal, but not $\{2,3\}$-maximal, see above. Unlike $\pi$-maximal subgroups,
$\pi$-submaxi\-mal ones have the following inductive property which resembles that of
$\pi$-Hall subgroups: if $A\trianglelefteqslant \trianglelefteqslant G$ then $H\cap A$ is $\pi$-submaximal in $A$ for every $\pi$-submaximal subgroup $H$ in $G$. The
Wielandt--Hartley theorem can be equivalently reformulated as follows: if $H$ is a $\pi$-submaximal subgroup of $G$ then $|N_G(H)/H|$ is not divisible by the primes in
$\pi$ \cite[5.4(a)]{80Wie}.

The concept of a $\pi$-submaximal subgroup proved to be quite useful. For example, the above-mentioned inductive property, along with the Wielandt--Hartley theorem,
shows that every $\pi$-submaximal (and, consequently, every $\pi$-maximal) subgroup of $G$ is uniquely, up to conjugacy, determined by its projections on the quotients
of any subnormal series of $G$ \cite[5.4(c)]{80Wie} (cf.~\cite[Ch.~5, Theorem~3.21]{86Suz}). More detailed information on $\pi$-submaximal subgroups can be found
in~\cite{18GuoRev1, 18GuoRev2}.

In this paper, we compare the behavior of $\pi$-maximal and $\pi$-submaximal subgroups under group homomorphisms.

Wielandt \cite[(14.2)]{94Wie607}, \cite[4.2]{80Wie} pointed out quite a general construction showing that if $\phi:G\rightarrow G_1$ is a homomorphism of groups and $H$
is a $\pi$-maximal subgroup of $G$, then $H^\phi$ is in general not $\pi$-maximal in $G^\phi$. Suppose that $X$ is a finite group that has more than one conjugacy class
of $\pi$-maximal subgroups, and let $Y$ be an arbitrary finite group. We consider the regular wreath product $X\wr Y$ and the natural epimorphism $G\rightarrow Y$.
Then every (not just $\pi$-maximal) $\pi$-subgroup of $Y=G^\phi$ coincides with $H^\phi$ for some $\pi$-maximal subgroup $H$ of $G$.

However, there exists a group $A$ with the following property: if some group $G$ contains a normal subgroup $N$ isomorphic to $A$ then the image $HN/H$ of every
$\pi$-maximal subgroup $H$ is $\pi$-maximal in $G/N$. Such is every $\pi$-group \cite[(12.4)]{94Wie607}, every $\pi'$-group \cite[(12.7)]{94Wie607}, where $\pi'$ is the
complement to $\pi$ in the set of all primes, and also (by induction) every $\pi$-separable group, i.\,e. a group having a subnormal series whose every factor is either
a $\pi$- or a $\pi'$-group, see. \cite[(12.9)]{94Wie607}, \cite[4.3]{80Wie}. In other words,
\begin{itemize}
 \item[$(*)$] {\it if $\phi:G\rightarrow G_1$ is a homomorphism of groups whose kernel is $\pi$-separable then $H^\phi$ is a $\pi$-maximal subgroup in $G^\phi$ for
     every $\pi$-maximal subgroup $H$ in $G$.}
\end{itemize}

An important remark about the behaviour of $\pi$-maximal subgroups under homomorphisms is the following consequence of \cite[Kap. III, Satz 3.8 and Satz 3.9]{67Hup}:
\begin{itemize}
 \item[$(**)$] {\it if $\phi:G\rightarrow G_1$ is a homomorphism of groups then every $\pi$-maximal subgroup in $G^\phi$ is the image under $\phi$ of some
     $\pi$-maximal subgroup of $G$.}
\end{itemize}

Properties $(*)$ and $(**)$ imply that if $\phi:G\rightarrow G_1$ is a homomorphism of groups whose kernel is $\pi$-separable then $H\mapsto H^\phi$ maps surjectively
the set of $\pi$-maximal subgroups of $G$ to the similar set of $G^\phi$ and induces a bijection between the conjugacy classes of $\pi$-maximal subgroups of $G$ and
$G^\phi$.

Do the $\pi$-submaximal subgroups have similar properties? In \cite[Proposition~9]{18GuoRev1}, it was proved that under a group homomorphism $\phi:G\rightarrow G_1$ the
image of a $\pi$-submaximal subgroup of $G$ is a $\pi$-submaximal subgroup of $G^\phi$ if the kernel of $\phi$ coincides with the $\pi$-radical, $\pi'$-radical,
nilpotent radical of $G$, and, more broadly, with the radical $G_\mathfrak{F}$, where $\mathfrak{F}$ is a Fitting class that consists of groups whose all maximal
$\pi$-subgroups are conjugate\footnote{Recall that a {\em Fitting class}  $\mathfrak{F}$ is a class of finite groups which contains with every group all its normal
subgroups, and if some group $G$ is the product of two of its normal subgroups belonging to $\mathfrak{F}$ then $G$ itself belongs to $\mathfrak{F}$. If $\mathfrak{F}$
is a Fitting class then every finite group has an {\em $\mathfrak{F}$-radical $G_\mathfrak{F}$}, i.\,e. the largest normal $\mathfrak{F}$-subgroup, which coincides with
the product of all normal $\mathfrak{F}$-subgroups of~$G$. Solvable groups, nilpotent groups, $\pi$-groups, $\pi'$-groups, $\pi$-separable groups are examples of Fitting
classes that consist of groups whose all maximal $\pi$-subgroups are conjugate.}. This property plays an important role in~\cite{18GuoRev}. H. Wielandt
posed \cite[Question~(g)]{80Wie} the problem of classifying $\pi$-submaximal subgroups in minimal nonsolvable groups\footnote{Recall that a group is {\em minimal
nonsolvable} if it is not solvable, but all its proper subgroups are solvable. If, additionally, the group is simple then it is {\em minimal simple}.}. Such a
classification for the minimal simple groups was obtained in \cite{18GuoRev}. The quotient $L=G/\Phi(G)$ of a minimal nonsolvable group $G$ by its Frattini subgroup
$\Phi(G)$ (which in this case coincides with the nilpotent radical) is a minimal simple group. As a consequence, the image of every $\pi$-submaximal subgroup of $G$
under the canonical epimorphism $G\to L$ is a $\pi$-submaximal subgroup of $L$~\cite[Proposition~1.2]{18GuoRev}. For every minimal simple group (the list of which was
obtained by J.Thompson \cite{68Tho}), all $\pi$-submaximal subgroups are found in~\cite[Tables~1--11]{18GuoRev}. The question of whether an analog of $(*)$ is true
remains: is a $\pi$-submaximal subgroup of~$L$ the image of some $\pi$-submaximal subgroup of~$G$? Note that most $\pi$-submaximal subgroups of $L$ are maximal
(see~\cite[Tables~1--11]{18GuoRev}), hence the above question is answered in the affirmative for such subgroups due to~$(*)$.

In the present paper, we construct examples showing that analogs of $(*)$ and $(**)$ for $\pi$-submaximal subgroups do not hold even in the case where the kernel of $\phi$ is
an abelian $\pi$-group. A counterexample $G$ to the analog of $(**)$ will be a minimal nonsolvable group and the corresponding homomorphism $\phi$ will be the canonical
epimorphism $G\rightarrow L=G/\Phi(G)$. We will therefore also construct an example of a $\pi$-submaximal subgroup in a minimal simple group $L$ that cannot be lifted to a
$\pi$-submaximal subgroup of some minimal nonsolvable group covering $L$.

\section{Preliminaries}

We will require the following result.

\begin{prop}\label{main}
Let $G$ be a finite group, let $\pi$ be a set of primes, and let $V$ be a unique minimal normal subgroup of $G$. Assume that $V$ is a $p$-group for $p\in \pi$,
$V\nleqslant \operatorname{Z}(G)$, and $L=G/V$ is nonabelian simple. Let $H$ be a $\pi$-submaximal subgroup of $G$ and let $G^*$ be a finite group of minimal order such
that $G\trianglelefteqslant \trianglelefteqslant G^*$ and there is a $\pi$-maximal subgroup $K$ of $G^*$ with $H=G \cap K$. Then $G\trianglelefteqslant G^*$ and
$C_{G^*}(G)=1$.
\end{prop}
\begin{proof} Denote $W=\langle V^g \mid g \in G^* \rangle$, i.\,e.
the normal closure of $V$ in $G^*$, and let $\overline{\phantom{a}}:G^*\to G^*/W$ be the canonical epimorphism. Note that $W$ is a $p$-group as a subgroup generated by
subnormal $p$-subgroups, see \cite[Satz 6.5(a)]{94Wie413}. Also, $W\leqslant K$, since $p\in \pi$ and $K$ is a $\pi$-maximal subgroup. Moreover, $W\cap G=V$ due to the
structure of $G$.

Denote $X=\langle G^g \mid g \in G^* \rangle$. By the minimality of $G^*$, we have $G^*=KX$. Note that $\overline{X}$ is a minimal normal subgroup of $\overline{G^*}$,
because it is the normal closure in $\overline{G^*}$ of the simple subnormal subgroup $\overline{G}\cong L$. This also yields $GW\trianglelefteqslant X$.

We show that every minimal normal subgroup of $G^*$ is a $\pi$-group and is therefore contained in $K$. Let $U$ be a minimal normal subgroup of $G^*$. If $U\leqslant W$,
the claim follows, since $W$ is a $p$-group. Otherwise, $U\cap W=1$ and $\overline U\cong U$. Then $\overline U$ is a minimal normal subgroup of $\overline{G^*}$.
\big(Indeed, if $M \leqslant U$  is such that $\overline{M}\trianglelefteqslant \overline{G^*}$ then
$$[M,G^*]\leqslant U \cap MW=M(U\cap W)=M,$$
i.\,e. $M\trianglelefteqslant G^*$, which implies $M=U$ or $M=1$.\big) It follows that either $\overline{U}=\overline{X}$ or $\overline{U}\cap \overline{X}=1$. The
former case, however, is impossible, since we would have in this case $X=UW$ and $[U,W]\leqslant U\cap W=1$, i.\,e. $X\cong W\times U$ which would contain no subgroup
isomorphic to $G$. Hence, the latter case holds, and because $G^*=KX$, we have $U\cong \overline{U}\leqslant \overline{K}$ is a $\pi$-group as claimed.

We now show that $W$ is the unique minimal normal subgroup of $G^*$. Again, let $U$ be a minimal normal subgroup of $G^*$. Then $U\leqslant K$ by the above. Suppose that
$U\cap V=1$. Then $U\cap G=1$ as $V$ is the unique proper normal subgroup of $G$. Let $\widetilde{\phantom{a}}:G^*\to G^*/U$ be the canonical epimorphism. We have
$G\cong \widetilde{G}\trianglelefteqslant\trianglelefteqslant \widetilde{G^*}$ and $\widetilde{K}$ is a $\pi$-maximal subgroup of $\widetilde{G^*}$, since $U\leqslant
K$. We have $GU\cap K=(G\cap K)U=HU$, and so $\widetilde{G}\cap\widetilde{K}=\widetilde{H}$. The minimality of $G^*$ excludes this possibility. Therefore, $U\cap V\ne 1$
and thus $V\leqslant U$, because $V$ is minimal normal in $G$. But then $W=\langle V^g \mid g \in G^* \rangle\leqslant U$, which gives $W=U$. This proves the uniqueness
of $W$ as a minimal normal subgroup of $G^*$.

We have $W=VC_W(G)$. Indeed, Clifford's theorem and the subnormality of $G$ in $G^*$ imply that $W$ is a completely reducible $\mathbb{F}_pG$-module. If $U$ is an
irreducible submodule of $W$ in and $U$ is not contained in $C_W(G)$ then
$$
U=[U,G]\leqslant U\cap G\leqslant W\cap G = V,
$$
where we have used the fact that $W$ normalizes $G$ by \cite[Theorem 2.6]{08Isa}. This implies that $W=VC_W(G)$.

We also have $C_{G^*}(X)=1$. Indeed, assuming the contrary we have $W\leqslant C_{G^*}(X)\trianglelefteqslant G^*$, since $W$ is the unique minimal normal subgroup of
$G^*$. Therefore, $V=W\cap G\leqslant C_G(G)=\operatorname{Z}(G)$ which contradicts the assumption that $V\nleqslant\operatorname{Z}(G)$.

Henceforth, denote $N=N_K(GW)$ and $G^0=N_{G^*}(GW)$.

We have $G^0=NX$. Indeed, $N\leqslant G^0$ by definition, and $X\leqslant G^0$, because $GW\trianglelefteqslant X$. Conversely, $G^0 \leqslant NX$, since $G^*=KX$.

We also have $H=N\cap G$. Indeed, $N\cap G\leqslant K\cap G=H$. Also, $H$ normalizes both $G$ and $W$, hence $H\leqslant N_K(GW)=N$.

We now show that $H=M\cap G$ for some $\pi$-maximal subgroup $M$ of $G^0$. Being a $\pi$-group, $N$ is contained in a maximal $\pi$-subgroup, say $M$, of $G^0$. We have
$G^0=NX=MX$ by the above. Let $1=g_1,\ldots,g_m$ be a right transversal of $N$ in $K$. This will also be a right transversal of $G^0$ in $G^*$, since $G^0=NX$, $G^*=KX$,
and $N\cap X=K\cap X$. Denote $M_i=(M\cap GW)^{g_i}$. For every $g\in K$, there exist $\sigma\in \operatorname{Sym}_m$ and elements $t_1,\ldots, t_m\in N$ such that
$g_ig=t_ig_{i\sigma}$. Therefore,
$$
M_i^g=(M\cap GW)^{g_ig}=(M\cap GW)^{t_ig_{i\sigma}}=(M\cap
GW)^{g_{i\sigma}}=M_{i\sigma},
$$
where we have used the fact that $t_i$ normalizes both $M$ and $GW$, since $t_i\in N$. It follows that $K$ normalizes the subgroup $M_X=\langle M_i\mid i=1,\ldots,
m\rangle$.

Since $W\leqslant K$, we have $W\leqslant N\leqslant M$ and $W\leqslant (M\cap GW)^{g_i}=M_i$ for every $i$. The fact that $g_1,\ldots,g_m$ form a right transversal of
$G^0$ in $G^*$ implies that the factors $M_i/W$ lie in distinct components of the minimal normal subgroup $X/W$ of $G^*/W$, and therefore, pairwise commute.
Consequently, $M_X/W$ is a $\pi$-group, hence so is $M_X$. The maximality of $K$ implies that $M_X\leqslant K$. On the other hand,
$$
M\cap G\leqslant M\cap GW = M_1\leqslant M_X\leqslant K.
$$
Now, we have $H=N\cap G\leqslant M\cap G\leqslant K\cap G=H$, and so $H=M\cap G$ as claimed.

In view of the minimality of $G^*$, we have $G^*=G^0$, $GW\trianglelefteqslant G^*$, $X=GW$, and $X/W\cong L$. It follows that $C_W(G)=C_W(X)\leqslant C_{G^*}(X)$. But
we showed that $C_{G^*}(X)=1$ and $W=VC_W(G)$. Consequently, $W=V$ and $X=G$. This implies $G\trianglelefteqslant G^*$. \end{proof}

\section{Examples}

In this section, we give two examples showing that $\pi$-submaximality is generally not preserved under homomorphic images or preimages. The following consequence of
Proposition~\ref{main} will be used.

\begin{cor}\label{crl}
Let $\pi=\{2,3\}$, let $L=\operatorname{GL}_3(2)$, and let $V$ be an elementary abelian group of order $8$. Suppose that $G$ is an upward extension of $V$ by $L$ such
that the conjugation action of $G$ induces on $V$ the natural $\mathbb{F}_2L$-module structure. Then no $2$-Sylow subgroup of $G$ is $\pi$-submaximal.
\end{cor}
\begin{proof}
Let $S$ be a $2$-Sylow subgroup of $G$ and suppose that it is $\pi$-submaximal. Let $\overline{\phantom{a}}:G\to L$ be the canonical epimorphism. Observe that $S$ is not
$\pi$-maximal, because neither is the $2$-Sylow subgroup $\overline{S}$ of $L$, inasmuch as $\overline{S}$ is contained in a subgroup  $M\leqslant L$ of order $24$. Let
$G^*$ be a finite group of minimal order such that $G\trianglelefteqslant \trianglelefteqslant G^*$ and there is a $\pi$-maximal subgroup $K$ of $G^*$ with $S=G \cap K$.
By Proposition \ref{main}, $G\trianglelefteqslant G^*$ and and $C_{G^*}(G)=1$, i.\,e. $G^*\leqslant \operatorname{Aut}(G)$. There are only two possibilities for $G$: it
is either the natural semidirect product $V\!\!:\!L$, or the unique\footnote{The nonsplit extension $V^{\,\textstyle \cdot}L$ is well known \cite[p. 120]{09Wils}. It
occurs as a maximal subgroup of the simple Chevalley group $G_2(p)$ for an odd prime $p$. The uniqueness follows from the fact that $\dim H^2(L,V)=1$.} nonsplit
extension $V^{\,\textstyle \cdot}L$. It can be shown using \cite{GAP4} that in both cases $\operatorname{Aut}(G)$ is an extension of $G$ by an outer automorphism
$\alpha$ of order $2$ which acts trivially on both $V$ and $L$. Consequently, we have either $G^*=G$ or $G^*=\operatorname{Aut}(G)$. Since $S$ is not $\pi$-maximal in
$G$, we cannot have $G^*=G$. Hence, $G^*=\operatorname{Aut}(G)$.

We extend ``$\ \overline{\phantom{a}}\ $'' to the canonical epimorphism $G^*\to G^*/V\cong L\times C$, where $C$ is cyclic of order $2$ generated by the image of
$\alpha$. Since $K$ is $\pi$-maximal in $G^*$ and $V$ is a $\pi$-subgroup, it follows that $\overline K$ is $\pi$-maximal in $G^*/V$. Also, we have $\overline{K}\cong
\overline{S} \times C$ by assumption. But this subgroup is not $\pi$-maximal in $G^*/V$ as it is properly contained in the $\pi$-subgroup $M\times C$, a contradiction.
\end{proof}

The first example shows that the image of a $\pi$-submaximal subgroup under an epimorphism $\phi$ whose kernel is an abelian $\pi$-group is not $\pi$-submaximal in
$\operatorname{Im}\phi$.

\medskip

{\em Example 1.} As in Corollary \ref{crl}, let $V$ be the natural $3$-dimensional $\mathbb{F}_2L$-module for $L=\operatorname{GL}_3(2)$, and let $V^*$ be its
contragredient module. The inverse-transpose automorphism $\gamma$ of $L$ of order $2$ permutes $V$ and $V^*$ and so naturally acts on $V\oplus V^*$. Consequently, the
semidirect product $G=(V\oplus V^*)\!:\!L$ can be extended to $G^*=G\!:\!\langle \gamma \rangle$. The $2$-Sylow subgroup of $G^*$ is $\pi$-maximal, where $\pi=\{2,3\}$.
Hence, the $2$-Sylow subgroup $S$ of $G$ is $\pi$-submaximal. Let $\overline{\phantom{a}}:G\to G/V^*$ be the canonical epimorphism. Corollary \ref{crl} implies that
$\overline{S}$ is not $\pi$-submaximal in $\overline{G}$, because $\overline{G}$ is an upward extension of $V$ by $L$.

\medskip

The second example shows that, for a homomorphism $\phi$ whose kernel is an abelian $\pi$-group, a $\pi$-submaximal subgroup
in $\operatorname{Im}\phi$ is not the image of any $\pi$-submaximal subgroup. Moreover, this example implies that there exists
a minimal nonsolvable group $G$ such that a $\pi$-submaximal subgroup in the minimal
simple group $G/\Phi(G)$ is the image of no $\pi$-submaximal subgroup of $G$.

\medskip

{\em Example 2.} Again, let $V$ be the natural module for $L=\operatorname{GL}_3(2)$ and let $G$ be the nonsplit extension $V^{\,\textstyle \cdot}L$. Then $G$ is a
minimal nonsolvable group. A $2$-Sylow subgroup $H$ of $L$ is $\pi$-submaximal by \cite{18GuoRev}, where $\pi=\{2,3\}$. However, there is no $\pi$-submaximal subgroup of
$G$ whose image in $L$ under the canonical epimorphism $G\to L$ would equal $H$, because such a subgroup would necessarily be $2$-Sylow in $G$, but no $2$-Sylow subgroup
of $G$ is $\pi$-submaximal by Corollary~\ref{crl}.

\medskip

{\em Remark.} We know of examples where $G$ has a normal abelian $\pi'$-subgroup $A$ such that the image under the canonical epimorphism $G\to G/A$
of a $\pi$-submaximal subgroup of $G$ is not $\pi$-submaximal in $G/A$, and conversely a $\pi$-submaximal subgroup of $G/A$ is the image of no
$\pi$-submaximal subgroup of $G$. A justification of these examples requires an analog of Proposition \ref{main} for the case $p\not\in \pi$.
We have obtained a proof of this analog which, however, uses the entire form of Wielandt's claim \cite[5.4(a)]{80Wie}, and therefore
we postpone the publication of our examples until after a proof of this claim has appeared in print, see footnote on p.~\pageref{skr}.



\begin{thebibliography}{10}

\bibitem{85Atlas}
J.~H. Conway, R.~T. Curtis, S.~P. Norton, R.~A. Parker, and R.~A. Wilson,
  \emph{{Atlas of finite groups}}, Clarendon Press, Oxford, 1985.

\bibitem{GAP4}
The GAP~Group, \emph{{GAP -- Groups, Algorithms, and Programming, Version
  4.10.1}}, 2019.

\bibitem{18GuoRev}
W.~Guo and D.~O. Revin, \emph{{Classification and properties of the
  $\pi$-submaximal subgroups in minimal nonsolvable groups}}, Bull. Math. Sci.
  \textbf{13} (2018), no.~2, 325--351.

\bibitem{18GuoRev1}
\bysame, \emph{{On maximal and submaximal ${\mathfrak X}$-subgroups}}, Algebra
  Logic \textbf{57} (2018), no.~1, 9--28.

\bibitem{18GuoRev2}
\bysame, \emph{{Pronormality and submaximal ${\mathfrak X}$-subgroups on finite
  groups}}, Commun. Math. Stat. \textbf{6} (2018), no.~3, 289--317.

\bibitem{28Hall}
P.~{Hall}, \emph{{A note on soluble groups.}}, {J. Lond. Math. Soc.} \textbf{3}
  (1928), 98--105.

\bibitem{71Har}
B.~Hartley, \emph{{A theorem of Sylow type for finite groups.}}, {Math. Z.}
  \textbf{122} (1971), 223--226 (English).

\bibitem{94Wie511}
\bysame, \emph{{Helmut Wielandt on the $\pi$-structure of finite groups}}, {in
  Helmut Wielandt, Mathematische Werke. Mathematical works. Ed. by B. Huppert
  and H. Schneider. Vol 1. Group theory} (1994), 511--516.

\bibitem{67Hup}
B.~Huppert, \emph{{Endliche {G}ruppen. {I}}}, {Grundlehren der Mathematischen
  Wissenschaften}, vol. 134, Springer-Verlag, Berlin, 1967.

\bibitem{08Isa}
I.~M. Isaacs, \emph{Finite group theory}, Graduate studies in mathematics,
  American Mathematical Society, 2008.

\bibitem{03She}
L.~A. {Shemetkov}, \emph{{Generalizations of Sylow's theorem.}}, {Sib. Math.
  J.} \textbf{44} (2003), no.~6, 1127--1132.

\bibitem{86Suz}
M.~Suzuki, \emph{{Group theory II}}, Grundlehren der Mathematischen
  Wissenschaften, vol. 248, Springer-Verlag, New York etc., 1986.

\bibitem{68Tho}
J.~G. {Thompson}, \emph{{Nonsolvable finite groups all of whose local subgroups
  are solvable.}}, {Bull. Am. Math. Soc.} \textbf{74} (1968), 383--437.

\bibitem{67Wie}
H.~Wielandt, \emph{{On the structure of composite groups}}, {Proc. Int. Conf.
  Theory Groups, Canberra 1965, 379-388}, 1967.

\bibitem{80Wie}
\bysame, \emph{{Zusammengesetzte Gruppen: H\"olders Programm heute}}, {Finite
  groups, Santa Cruz Conf. 1979, Proc. Symp. Pure Math. 37, 161--173}, 1980.

\bibitem{94Wie413}
\bysame, \emph{{Subnormale Untergruppen endlicher Gruppen}}, {Lecture notes
  prepared by M. Selinka, Math. Inst. Univ. T\"ubingen (1971); in Helmut
  Wielandt, Mathematische Werke. Mathematical works. Ed. by B. Huppert and H.
  Schneider. Vol 1. Group theory} (1994), 413--479.

\bibitem{94Wie607}
\bysame, \emph{{Zusammengesetzte Gruppen endlicher Ordnung}}, {Lecture notes
  prepared by J. Breuninger, Math. Inst. Univ. T\"ubingen (1963/64); in Helmut
  Wielandt, Mathematische Werke. Mathematical works. Ed. by B. Huppert and H.
  Schneider. Vol 1. Group theory} (1994), 607--655.

\bibitem{09Wils}
R.~A. Wilson, \emph{{The finite simple groups}}, {Graduate Texts in
  Mathematics}, vol. 251, Springer-Verlag, London, 2009.

\end{thebibliography}
\end{document}